 \documentclass[12pt]{amsart}

\usepackage{amsmath,graphics,comment,color}
\usepackage[dvips]{graphicx}
\usepackage{amsfonts,amssymb}

  \includecomment{inlong} \excludecomment{inshort}
 \includecomment{inlow} \excludecomment{inhigh}
\usepackage{etoolbox}
\usepackage[dvips]{graphicx}
\usepackage{hyperref,color,mathdots,accents,extarrows,bigints,enumitem,stmaryrd}
\usepackage[all]{xypic}
\usepackage[ansinew]{inputenc}
\usepackage{amsmath,amsfonts,amssymb}
\usepackage[all,graph]{xy}
\usepackage[shortcuts]{extdash}

\theoremstyle{plain}
\newtheorem*{theorem*}{Theorem}
\newtheorem*{lemma*} {Lemma}
\newtheorem*{corollary*} {Corollary}
\newtheorem*{proposition*} {Proposition}
\newtheorem{theorem}{Theorem}[section]
\newtheorem{lemma}[theorem]{Lemma}
\newtheorem{corollary}[theorem]{Corollary}
\newtheorem{proposition}[theorem]{Proposition}

\theoremstyle{remark}
\newtheorem*{remark}{Remark}
\newtheorem*{definition}{Definition}

\theoremstyle{definition}

\def \Z {\mathbf{Z}}
\def \C {\mathbf{C}}

\newcommand{\smfrac}[2]{\mbox{\footnotesize$\displaystyle\frac{#1}{#2}$}} 
\newcommand{\tmfrac}[2]{\mbox{\large$\frac{#1}{#2}$}}

\def\Z{\Bbb{Z}}
\def\C{\Bbb{C}}

\def\part{\partial}

\def\a{\alpha}

\def\bp{\begin{pmatrix}}

\def\sm{\setminus}
\def\ep{\end{pmatrix}}
\def\bn{\begin{enumerate}}

\def\en{\end{enumerate}}
\def\ba{\begin{array}}
\def\ea{\end{array}}

\def\a{\alpha}

\def\ti{\widetilde}

\def\fr12{\frac{1}{2}}

\def\Char{\mbox{char}}

\def\cmtbf#1{} \def\cmt#1{}

\begin{document}

\title{An elementary proof of the group law for elliptic curves}

\author{Stefan Friedl}
\address{Fakult\"at f\"ur Mathematik\\ Universit\"at Regensburg\\93040 Regensburg\\   Germany}
\email{sfriedl@gmail.com}

\date{\today}
\begin{abstract}
We give an elementary proof of the group law for elliptic curves using explicit formulas.
 \end{abstract}
\maketitle


\section{Introduction}

In this short note we give an elementary proof of the well--known fact that the addition of points on an elliptic curve defines a group structure. We only use explicit and very well--known formulas for the coordinates of the addition of two points.
Even though the arguments in the proof are elementary, making this approach work requires several intricate arguments and elaborate computer calculations. The approach of this note was used by Laurent Th\'ery  \cite{Th07} to give a formal proof of the group law for elliptic curves in the formal proof management system Coq.
Some of the ideas of our approach were also  used by Thomas Hales~\cite{Ha16} to give an  elementary computational proof of the
group
law for Edwards elliptic curves.

In the following $K$ will denote an algebraically closed field with $\Char(K)>3$.
An elliptic curve is defined as  a pair $(E,O)$ where $E$ is a smooth algebraic
curve of genus one  and $O\in E$ a point. The following proposition tells us in particular that the precise definition of an elliptic curve is irrelevant to us since we describe any elliptic curve in an elementary way.

\begin{proposition}\cite[Prop.~3.1]{Si86}
Let $E$ be an elliptic curve.
Then there exist $a,b\in K$ with $4a^3+27b^2\ne 0$ and an isomorphism
of curves
\[ \phi\colon E \,\,\to\,\, E_{a,b}:=\{ [x:y:z]\in \mathbb{P}(K)^2 \,|\, zy^2=x^3+axz^2+bz^3 \} \]
such that $\phi(O)=[0:1:0]$. Conversely, for any $a,b\in K$ with $4a^3+27b^2\ne 0$
the variety $(E_{a,b},[0:1:0])$ is an elliptic curve.
\end{proposition}

\begin{corollary}
Let $E$ be an elliptic curve.
Then there exist $a,b\in K$ with $4a^3+27b^2\ne 0$ and a bijection
\[ \phi\colon  E \to E_{a,b}^{\operatorname{affine}}:=\{ (x,y)\in K^2 | y^2=x^3+ax+b \} \cup \{O\} \]
such that $\phi(O)=O$.
\end{corollary}

In the following we will take the affine point of view, i.e.\ an
elliptic curve $E$ will mean the set $E_{a,b}^{\operatorname{affine}}$ for some $a,b\in K$ with $4a^3+27b^2\ne 0$.
The point $O$ is called the ``point at infinity''.
For points $A,B,C,\dots \in E\sm \{O\}$ we will write
$A=(x_A,y_A), B=(x_B,y_B), C=(x_C,y_C),\dots$ for the coordinates.

\begin{definition}
We define $+\colon E\times E\to E, (A,B)\mapsto A+B$ as follows.
\bn
\item[(a)]
We set $A+O=O+A:=O$ for all $A$.
\item[(b)]
If $(x_A,y_A)=(x_B,-y_B)$, then $A+B:=O$.
\item[(c)] If $(x_A,y_A)\ne (x_B,-y_B)$, then we define $A+B:=(x_{AB},y_{AB})$ where
\begin{eqnarray} \label{addpq}
      x_{AB} & := & \a(A,B)^2-x_A-x_B \nonumber \\[-6.5pt]
      &&\\[-6.5pt]
\vspace*{-2cm} y_{AB} & := & -y_A+\a(A,B)(x_A-x_{AB}) \nonumber
\end{eqnarray}
with $\a(A,B)=\tmfrac{y_A-y_B}{x_A-x_B}$ if $x_A\ne x_B$
and $\a(A,B)=\tmfrac{3x_A^2+a}{2 y_B}$ if $x_A=x_B$.
\en
\end{definition}

\begin{remark}
This  definition of the ``addition'' $+$ can be described geometrically. More precisely,
if $A,B \ne O$ are two points on $E$ with  $A\ne B$, then we first consider the line $g$ through $A$ and $B$, it intersects the curve $E$ in a third point $C$ and we define $A+B:=-C$ where $-C$ is the reflection of $C$ about the $x$-axis.
Similarly, if $A\ne O$ is a point on $E$, then $A+A$ is defined by considering the tangent $h$ to $E$ at the point $A$, it intersects the curve $E$ in a third point $C$ and we define $A+A:=-C$.
\begin{figure}[h]
\begin{center}
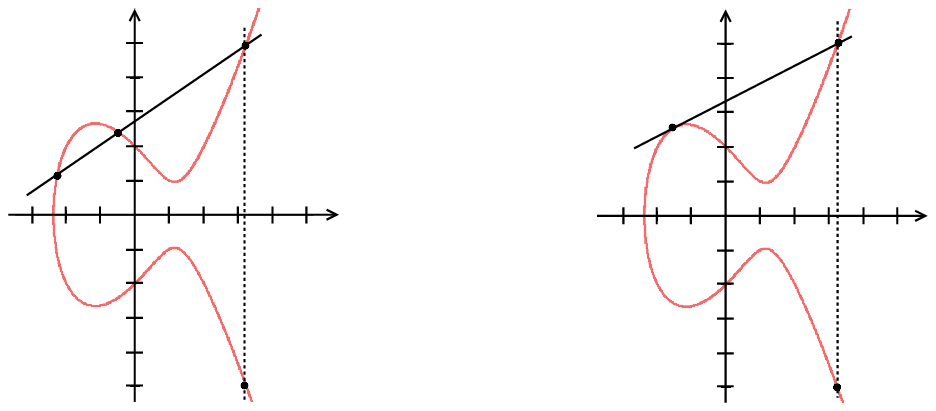
\label{fig:elliptic-curve-addition}
\end{center}
\end{figure}
\end{remark}

The following theorem is the main result of this paper.

\begin{theorem}
$(E,+)$ is an abelian group
with neutral element $O$.
\end{theorem}

This theorem is of course well known. For example in \cite{Si86} it is shown that the above
structure is isomorphic to the group structure of $\mbox{Pic}^0(E)$, the Picard group of $E$. In particular
$(E,+)$ forms a group. A more geometric argument for the statemant that $(E,+)$ defines a group
is given in \cite[p. 87]{Ku95} or \cite{Hu87}.
Perhaps the most elementary proof can be found for $K=\C$ using the Weierstrass function (cf. \cite{La78}).
By the Lefschetz principle this shows the theorem for any algebraically closed field of characteristic zero.

We will give a completely elementary proof, just using the above explicit definition
of the group structure through formulas.
It must have always been  clear that such a proof exists, but it turns out that
this direct proof is more difficult than one might have imagined initially.
Many special cases have to be dealt with separately and some are non--trivial.
Furthermore it turns out that the explicit computations in the proof are very hard.
The verification of some identities took several hours on a modern computer; this proof
could not have been carried out before the 1980's.\\

\emph{Added in proof. Here by a ``modern computer'' I mean a PC from 1998 with 16Mb RAM. }

\subsection*{Note}
This elementary proof was part of my undergraduate thesis written in 1998
at the University of Regensburg written under the supervision of Ernst Kunz. Several years ago I posted this extract from my thesis on my webpage. Slightly to my own surprise (and surely thanks to google) it has been cited on several occasions~\cite{Th07,TH07,Ha16,Ru,Ru17}.

\subsection*{Acknowledgment.} 
I am grateful to Martin Kreuzer and Armin R\"ohrl for helping me with the computer calculations. 
I am especially grateful to Ernst Kunz for suggesting this problem to me. I also would like to use this opportunity to thank Ernst Kunz for his  support in the early stages of my mathematical career. Without his help I would never have become a mathematician.

\section{Proof of the associativity law for elliptic curves}
In the following let  $E$ be a fixed elliptic curve.
It is clear that ``+'' is commutative, that $O$ is a neutral element and that
the inverse element for $A=(x_A,y_A)$ is given by $-A:=(x_A,-y_A)$.
The only difficult part is to show that ``+'' is in fact associative.
This proof will require the remainder of this paper.
\\

Throughout this section we will use the following facts which follow immediately from the definition.
\bn
\item For  $A=(x_A,y_A) \in E \setminus \{O\}$ we have  $A+A=O$ if and only if  $y=0$.
\item If $A,B \in E \sm \{O\}$ and $x_A=x_B$, then $A=B$ or $A=-B$.
\en

Except for three special cases the operation ``+'' is given by
Formula  (\ref{addpq}).
In Section~\ref{section31} we will show the associativity in three out of four cases in which
addition is given by either of the two formulas.
This will be done using explicit calculations.

In Section ~\ref{sectionsp} we will prove several lemmas, which we will use in Section
\ref{sectionproof} to give the proof in the general case.

\subsection{Proof for the generic cases} \label{section31}

In this section we consider the cases in which only Equation (\ref{addpq}) is being
used in the definitions of $(A+B)+C$ and $A+(B+C)$.

\begin{lemma} \label{claimabc}
Let $A,B,C \in E \setminus \{O\}$.
If     $A \neq \pm B$, $B \neq \pm C$, $A+B \neq \pm C$
and $B+C \neq \pm A$,
then
\[ (A+B)+C=A+(B+C). \]
\end{lemma}

\begin{proof}
Write $(x_1,y_1):=(A+B)+C$ and $(x_2,y_2):=A+(B+C)$.
Let
\[ \ba{rclrcl}
\alpha&:=&\tmfrac{y_B-y_A}{x_B-x_A}, &\beta&:=&\tmfrac{y_A+y_C-\alpha(2x_A+x_B-\alpha^2)}{x_A+x_B+x_C-\alpha^2}, \\
\gamma&:=&\tmfrac{y_B-y_C}{x_B-x_C}, &\tau&:=&\tmfrac{y_A+y_B-\gamma(2x_B+x_C-\gamma^2)}{x_A+x_B+x_C-\gamma^2}. \ea
\]
Using Equation (\ref{addpq})  we get
\[
\ba{rclrcl}
 x_1&=&\beta^2+x_A+x_B-x_C-\alpha^2, &  y_1&=&-y_C+\beta(2x_C-x_A-x_B-\beta^2+\alpha^2), \\
 x_2&=&\tau^2+x_B+x_C-x_A-\gamma^2, & y_2&=&-y_A+\tau(2x_A-x_B-x_C-\tau^2+\gamma^2). \ea \]
Setting
\[
\ba{rclrcl}
\widetilde\alpha&:=&y_B-x_A, &\widetilde\beta&:=&(y_A+y_C)(x_B-x_A)^3-\widetilde\alpha ((2x_A+x_B)(x_B-x_A)^2-\widetilde\alpha^2), \\
\widetilde\gamma&:=&y_B-y_C, &\widetilde\tau&:=&(y_A+y_B)(x_B-x_C)^3-        \widetilde\gamma((2x_B+x_C)(x_B-x_C)^2-\widetilde\gamma^2),\\
\widetilde\eta&:=&x_B-x_A, &\widetilde\mu&:=&x_B-x_C.
\ea \]
one can show that $x_1=x_2$ is equivalent to
\vspace{0.4cm}\\
\hspace*{0.5cm}$(\widetilde\beta^2(x_B-x_C)^2+(((2x_A-2x_C)(x_B-x_C)^2
    +\widetilde\gamma^2)(x_B-x_A)^2-\widetilde\alpha^2(x_B-x_C)^2)$\\[0.1cm]
\vspace{0.1cm}
\hspace*{0.5cm}$((x_A+x_B+x_C)(x_B-x_A)^2-\widetilde\alpha^2)^2)((x_A+x_B+x_C)
     (x_B-x_A)^2-\widetilde\gamma^2)^2$\\[0.1cm]
\vspace{0.1cm}
$-\widetilde\tau^2((x_A+x_B+x_C)(x_B-x_A)^2-\widetilde\alpha^2)^2(x_B-x_A)^2=0$\\[0.3cm]
and $y_1=y_2$ is equivalent to
\vspace{0.4cm}\\
$\hspace*{0.5cm}(y_A-y_C)((x_A+x_B+x_C)\widetilde\eta^2-\widetilde\alpha^2)^3
((x_A+x_B+x_C)\widetilde\mu^2-\widetilde\gamma^2)^3\widetilde\eta^3\widetilde\mu^3$\\[0.1cm]
\vspace*{0.1cm}
$+\widetilde\beta(((2x_C-x_A-x_B)\widetilde\eta^2+\widetilde\alpha^2)((x_A+x_B+x_C)
 \widetilde\eta^2-\widetilde\eta^2)^2-\widetilde\beta^2)$\\
\vspace{0.1cm}
\hspace*{1cm}$((x_A+x_B+x_C)\widetilde\mu^2-\widetilde\gamma^2)^3\widetilde\mu^3$\\
\vspace{0.1cm}
$-\widetilde\tau(((2x_A-x_B-x_C)\widetilde\mu^2+\widetilde\gamma^2)
  ((x_A+x_B+x_C)\widetilde\eta^2-\widetilde\gamma^2)^2-\widetilde\tau^2)$\\
\vspace{0.1cm}
\hspace*{1cm}$((x_A+x_B+x_C)\widetilde\eta^2-\widetilde\alpha^2)^3\widetilde\eta^3=0.$\\[0.3cm]
By abuse of notation we now consider the equations
over the polynomial ring 
\[ P\,:=\,\Z[x_A,x_B,x_C,y_A,y_B,y_C,a,b].\]
It suffices to show that the equalities hold in $P/I$
where 
\[ I\,:=\,(y_A^2-x_A^3-ax_A-b,y_B^2-x_B^3-ax_B-b, y_C^2-x_C^3-ax_C-b).\]
This is equivalent to showing that both left hand sides lie in $I$.
This was shown using
the commutative algebra package `CoCoA' \cite{cocoa}.
\end{proof}

In a very similar way one can show the following two lemmas.

\begin{lemma} \label{claimaac}
If $A,B \ne O, A \neq -A, A \neq \pm B, A+A\not=\pm B$ and
$A+B\not=\pm A$, then
 \[ (A+A)+B=A+(A+B). \]
\end{lemma}

\begin{lemma} \label{claimaaaa}
If $A \ne O, A \neq -A$, $A+A \neq -(A+A)$, $(A+A)+A \neq \pm A$
and $A+A \neq \pm A$, then
\[ (A+A)+(A+A)=A+(A+(A+A)). \]
\end{lemma}

The next step would be to show that under the above restrictions, we have
\[(A+B)+(A+B)\,\,=\,\,A+(B+(A+B)).\]
We will show this without reverting to explicit computations in the proof of Theorem~\ref{thmabc}.

\subsection{Proof of basic properties} \label{sectionsp}

\begin{lemma} \label{claim2}
For $A,B\in E$ we have
  \[ -A-B=-(A+B).\]
\end{lemma}

\begin{proof}
The cases $A=O, B=O$ and $A=-B$ are trivial.
In the other cases the lemma follows from an easy calculation
using Equation (\ref{addpq}).
\end{proof}

\begin{lemma} \label{claimpmb}
Let $A,B\in E$.
If $A+B=A-B$ and $A \neq -A$, then $B=-B$.
\end{lemma}

\begin{proof}
The cases $A=O$ respectively $B=O$ are trivial. If $A=\pm B$, then $B=-B$ follows easily from the uniqueness
of the inverse element.
So assume that $A,B \ne O, A \ne \pm B$.
Using Equation (\ref{addpq}) we get
\[
\left(\smfrac{y_B-y_A}{x_B-x_A}\right)^2-x_A-x_B\,\,=\,\,\left(\smfrac{-y_B-y_A}{x_B-x_A}\right)^2-x_A-x_B. \]
This simplifies to  $ -2y_Ay_B= 2y_Ay_B$.
 Since $A \neq -A$ it follows that $y_A \neq 0$.
 We get $y_B=0$ since $\Char(K)>3$, hence $B=-B$.
\end{proof}

\begin{lemma}[Uniqueness of the neutral element]\label{claimnull}
Let $A,B\in E$. If $A+B=A$, then $B=O$.
\end{lemma}

\begin{proof}
The cases $A=O$ and $A=-B$ are trivial.
Now assume that $A \neq O, A \neq -B$. Assume that $B\ne O$.
Write   $(x_C,y_C):=A+B=A=(x_A,y_A)$.
It follows from Formula  (\ref{addpq})   that
\[ y_A=y_C=-y_A+\alpha(P,Q) \underbrace{(x_A-x_C)}_{=0}=-y_A\]
i.e.\ $y_A=0$, therefore $A=-A$. It follows that
\[ A+B=A=-A=-A-B=A-B.\]
According to Lemma~\ref{claimpmb} this means that $B=-B$, i.e.\ $y_B=0$.
In particular $A \neq B$, because otherwise we would get $B=A=A+B=A+A=A-A=O$.
According to Formula  (\ref{addpq}) we get
\[ x_A=x_C=\left(\smfrac{y_B-y_A}{x_B-x_A}\right)^2-x_A-x_B=-x_A-x_B \]
since $y_A=y_B=0$.
Therefore $x_A$ and $x_B=-x_A-x_A$ are zeros of the polynomial $P:=X^3+aX+b$.
It follows that $x_0=-x_A-x_B=x_A$ is the third zero since the second highest coefficient of $P$ is zero.
In particular $x_A$ is a zero of degree 2. This leads to a contradiction,
since we assumed that the discriminant $4a^3+27b^2$ of the polynomial $X^3+aX+b$ is non--zero,
i.e.\ the polynomial has distinct zeros.
\end{proof}

\begin{lemma} \label{claimaaa}
Let $A\in E$. If $A \neq -A$ and $A+A \neq -A$,
then
$(A+A)-A=A$.
\end{lemma}

\begin{proof}
The cases $A=O$ and $A+A=O$ are trivial. The general case follows from an easy computation
using Equation (\ref{addpq}).
\end{proof}

\begin{lemma} \label{claim5}
Let $A,B\in E$. If $A+B=-A$, then $B=-A-A$.
\end{lemma}

\begin{proof}
The cases $A=O$, $B=O$, $A=B, A=-B$ are trivial.
If $A=-A$, then $-A+B=-A$. Using Lemma~\ref{claimnull} it follows that $B=O$. Hence $B=O=A-A=-A-A$.
Now assume that $A\neq \pm B$ and $A \neq -A$, $A,B \ne O$.
From $-A=A+B$ it follows that
 \[ x_A\,\,=\,\,\left(\smfrac{y_A-y_B}{x_A-x_B}\right)^2-x_A-x_B   \]
 which is equivalent to $2y_Ay_B=y_A^2+ax_B+b-2x_A^3+3x_A^2x_B$. Squaring both sides we get
 \[4x_B^3y_A^2-x_B^2(3x_A^2+a)^2+x_B(2a^2x_A+6x_A^5-12bx_A^2)-(y_A^2-b)^2+4ax_A^4+8bx_A^3=0\]
      which in turn is equivalent to
      \[\left(x_B-\left(\left(\smfrac{3x_A^2+a}{2y_A}\right)^2-2x_A\right)\right)
                (x_B-x_A)^2=0.\]
      Since we excluded the case $x_A=x_B$ we get
      \[x_B=\left(\smfrac{3x_A^2+a}{2y_A}\right)^2-2x_A\]
      i.e.\  $B=A+A$ or $B=-(A+A)=-A-A$.
      If $A+A=-A$, then $B=\pm(A+A)=\pm A$. Hence  $A+A \neq -A$.
      By Lemma~\ref{claimaaa} it follows that $B=-A-A$ is a solution for the equation $A+B=-A$.
      If $B=A+A$ is also a solution, then
      \[A+B\,\,=\,\,A+(A+A)\,\,=\,\,-A\,\,=\,\,A-(A+A)\,\,=\,\,A-B.\] Since $A \neq -A$ it follows from Lemma~\ref{claimpmb}, that $B=-B$. Therefore we obtain that $B=-B=-A-A$.
\end{proof}

\begin{lemma}[Cancelation rule] \label{claim3}
Let $A,B,\widetilde{B} \in E$.
If $A+B=A+\widetilde B$, then $B=\widetilde B$.
\end{lemma}

\begin{proof}
If $A=O$, then immediately $B=\widetilde B$.
The cases $B=O$ and $A+B=O$ follow immediately from the uniqueness of the neutral element
(Lemma~\ref{claimnull}) and the uniqueness of the inverse element.
If $A+B=A+\widetilde B=-A$, then using~\ref{claim5} we see that $B=-A-A$ and $\widetilde B=-A-A$.

We therefore can assume that $A,B,\ti{B} \ne O$
 and $A+B=A+\widetilde B \ne O, A+B\ne -A$. Writing $A+B=A+\widetilde B=:(x_C,y_C)$
we get
      \[
      \begin{array}{rcccl}
       x_C&=&\alpha(A,B)^2-x_A-x_B&=&\alpha(A,\ti{B})^2-x_A-\widetilde x_B \\
       y_C&=&-y_A+\alpha(A,B)(x_A-x_C)&=&-y_A+\alpha(A,\ti{B})(x_A-x_C).\\
      \end {array}
      \]
      From $A+B \neq \pm A$ it follows that $x_A \neq x_C$, from the second equation we get $\a(A,B)=\a(A,\ti{B})$.
      Using the first equation we get  $x_B=\widetilde x_B$, i.e.
      $B=-\widetilde B$, or $B=\widetilde B$.
      We consider the following two cases:
      \bn
      \item If $A=-A$, then $B, \widetilde B \neq -A=A$, hence
            \[ \smfrac{y_B-y_A}{x_B-x_A}\,\,=\,\,\alpha(A,B)\,\,=\,\,\alpha(A,\ti{B})\,\,=\,\,
               \smfrac{\widetilde y_B-y_A}{\widetilde x_B-x_A}. \]
            Since $x_B=\widetilde x_B$ we get $y_B=\widetilde y_B$, therefore
            $B=\widetilde B$.
      \item If $A \neq -A$, then assume that $B=-\widetilde B$. It follows that
            $A+B=A+\widetilde B=A-B$.
            By Lemma~\ref{claimpmb} $B=-B$, since $A \neq -A$. Therefore $B=\widetilde B$.\qedhere
      \en
\end{proof}

\begin{lemma} \label{claim4}
For any $A,B \in E$ we have
\[ (A+B)-B=A. \]
\end{lemma}

\begin{proof}
The cases $A=O, B=O$ respectively $A=-B$ are trivial.
The case $A=B$ follows from Lemma~\ref{claimaaa}.
If $A+B=-B$ and $A\ne -B$, then we obtain from  Lemma~\ref{claim5} that $A=-B-B$, hence $(A+B)-B=-B-B=A$.

Now assume that $A,B\ne O, A\ne \pm B, A+B\ne -B$.
This case follows from an explicit computation using Equation (\ref{addpq}).
\end{proof}

\begin{corollary} \label{cor4}
Let $A,B,C\in E$. If $ A+B=C$, then $A=C-B$.
\end{corollary}

\begin{proof}
From Lemma~\ref{claim4} we get $A+B=A+(C-A)$, the corollary now follows from Lemma~\ref{claim3}.
\end{proof}

\subsection{Completion of the proof} \label{sectionproof}

\begin{lemma} \label{lemmavabc}
Let $A,B,C\in E$.
Assume that
\bn[font=\normalfont]
\item $(A+B) \neq C$ and $A \neq (B+C)$, or
\item $A=B$, or $B=C$, or $A=C$, or
\item $O \in \{A,B,C,A+B,B+C,(A+B)+C,A+(B+C)\}$,
\en
then
\[ (A+B)+C\,\,=\,\,A+(B+C). \]
\end{lemma}

\begin{proof}
The cases $A=O$, $B=O$, $C=O$ and $A=C$ are trivial.
The cases $A=-B$ and $C=-B$ follow immediately from Lemma~\ref{claim4}.
If $A+B=-C$, then by Lemma~\ref{claim4}
      \[ (A+B)+C=O=A-A=A+(B+(-B-A))=A+(B+C). \]
The case $B+C=-A$ works the same way. We thus established part (3) of the lemma.

We can therefore assume that $A,B,C \ne O$, $A \neq C, B\ne -A,-C, A+B \not=-C$ and $B+C\not=-A$.

If $A=B$, then we have to show that $(A+A)+C=A+(A+C)$.
This follows from Lemmas~\ref{claimaac} ($C \neq A+A$) and~\ref{claimaaaa} ($C=A+A$).
The case $B=C$ again works the same way.
This shows part (2) of the lemma.

The remaining cases of part (1) now follow immediately from Lemma~\ref{claimabc}.
\end{proof}

\begin{theorem} \label{thmabc}
Let $A,B,C \in E(K)$. Then
\[ (A+B)+C\,\,=\,\,A+(B+C) \]
\end{theorem}

\begin{proof}
By
Lemma~\ref{lemmavabc}.
we only have to prove the theorem for $A,B,C$ with $A+B=C$ or $B+C=A$.
Clearly it is enough to consider only the case $A+B=C$.
We therefore have to show that
\[(A+B)+(A+B)=A+(B+(A+B)).\]
By Lemma~\ref{lemmavabc} we can assume that $A,B,C,A+B,B+C,(A+B)+C,A+(B+C) \neq  O$ and that 
$A,B,C$ are pairwise different.

If $(A+B)+(A+B)=-A$, then $A+B=(-B-A)-A$ by Corollary~\ref{cor4}.
Furthermore $(-B-A)-A=-B+(-A-A)$ by the second part of Lemma~\ref{lemmavabc}, hence
       $A+B=-B+(-A-A)$. We get
      \begin{eqnarray}
      A+(B+(A+B))&\hspace*{-0.2cm}=&\hspace*{-0.2cm}A+(B+(-B+(-A-A)))=A+(-A-A)= \nonumber \\
                 &\hspace*{-0.2cm}=&\hspace*{-0.2cm}-A=(A+B)+(A+B). \nonumber
      \end{eqnarray}
If $(A+B)+(A+B)\neq -A$, then $((A+B)+(A+B))-A=(A+B)+((A+B)-A)$ by
the second part of Lemma~\ref{lemmavabc}. Hence
      \begin{eqnarray}
      ((A+B)+(A+B))-A&\hspace*{-0.2cm}=&\hspace*{-0.2cm}(A+B)+((A+B)-A)=(A+B)+B= \nonumber \\
                     &\hspace*{-0.2cm}=&\hspace*{-0.2cm}(A+(B+(A+B)))-A. \nonumber
      \end{eqnarray}
      From Lemma~\ref{claim3} it follows, that $(A+B)+(A+B)=A+(B+(A+B))$.
\end{proof}

\end{document}